\documentclass[conference]{IEEEtran}

\IEEEoverridecommandlockouts
\usepackage{cite}
\usepackage{amsmath,amssymb,amsfonts,amsthm}
\usepackage{graphicx}
\usepackage{textcomp}
\usepackage{xcolor,soul,framed} 
\usepackage{mdwtab}
\usepackage{eqparbox}
\usepackage{array}
\usepackage[caption=false,font=normalsize,labelfont=sf,textfont=sf]{subfig}
\usepackage{textcomp}
\usepackage{stfloats}
\usepackage{url}
\usepackage{verbatim}
\usepackage{cite}
\usepackage{multirow}
\usepackage{indentfirst}
\usepackage{booktabs}
\usepackage{breqn}
\usepackage{bm}
\usepackage{makecell}
\usepackage[ruled]{algorithm2e}
\usepackage{url}
\usepackage{hyperref}
\usepackage{xcolor}
\usepackage[normalem]{ulem}

\def\BibTeX{{\rm B\kern-.05em{\sc i\kern-.025em b}\kern-.08em
T\kern-.1667em\lower.7ex\hbox{E}\kern-.125emX}}
\usepackage[ruled]{algorithm2e}
\usepackage{algpseudocode}
\newtheorem{theorem}{\textbf{Theorem}}

\newtheorem{proposition}[theorem]{\textbf{Proposition}}

\usepackage{url}
\usepackage{hyperref} 
\usepackage{xcolor}

\hypersetup{
    colorlinks=true, 
    linkcolor=blue, 
    filecolor=black, 
    urlcolor=blue,
    citecolor=blue, 
    linkbordercolor={1 1 1}, 
    pdfborder={0 0 0} 
}
\graphicspath{ {./plots/} }
\begin{document}

 \title{Nonparametric Kernel Regression for Coordinated Energy Storage Peak Shaving with Stacked Services}
\author{\IEEEauthorblockN{Emily Logan, Ning Qi, Bolun Xu} 
\IEEEauthorblockA{\textit{Department of Earth and Environmental Engineering} \\
\textit{Columbia University} \\
\textit{New York, NY 10027, USA}\\
\{el3350, nq2176, bx2177\}@columbia.edu}
\vspace{-0.5cm}
\thanks{This work was partly supported by the Department of Energy under Grant No. DE-EE0011385 and by the National Science Foundation under award ECCS-2239046. }
}

\maketitle

\begin{abstract}

Developing effective control strategies for behind-the-meter energy storage to coordinate peak shaving and stacked services is essential for reducing electricity costs and extending battery lifetime in commercial buildings. This work proposes an end-to-end, two-stage framework for coordinating peak shaving and energy arbitrage with a theoretical decomposition guarantee. In the first stage, a non-parametric kernel regression model constructs state-of-charge trajectory bounds from historical data that satisfy peak-shaving requirements. The second stage utilizes the remaining capacity for energy arbitrage via a transfer learning method. Case studies using New York City commercial building demand data show that our method achieves a 1.3 times improvement in performance over the state-of-the-art forecast-based method, achieving cost savings and effective peak management without relying on predictions.
\end{abstract}

\section{Introduction}

The rapid growth of electricity demand, along with limited grid infrastructure, has intensified the need for effective demand-side management. Large commercial and industrial consumers are facing significant electricity costs, including both real-time electricity charges and monthly peak demand charges, with the latter being the more substantial. Behind-the-meter battery energy storage systems (BESS) have thus become an attractive solution for peak shaving and electricity cost reduction, while also mitigating grid stress~\cite{li2023dynamic,Mirletz2023}.

Peak shaving is challenging, as it requires decision-making under uncertainty over monthly horizons. Many existing methods rely on demand and price forecasting and uncertainty models such as regression~\cite{deSalis2014energy}, neural-networks~\cite{zhang2023novel}, and chance-constrained models~\cite{qi2023chance}. However, reliable forecasting or uncertainty models at a monthly resolution remain computationally intensive~\cite{kaut2024handling, cortes2025economic, mo2021optimal}. The complexity increases significantly when peak shaving is integrated with other grid services, such as energy arbitrage and frequency regulation~\cite{shi2017using}, requiring careful coordination to balance the competing objectives. Other approaches, including dynamic programming~\cite{engels2020optimal} and portfolio optimization~\cite{qi2023portfolio}, have also been employed to optimize stacked storage services, but these methods primarily focus on daily peak shaving and struggle with computational burden when adapting to real-time uncertainties or extending to monthly horizons. This limits their practicality for smaller businesses and non-profits, where costs may exceed the potential savings.

We propose an end-to-end, two-stage framework to coordinate peak shaving with stacked services such as energy arbitrage. Unlike conventional methods that rely on forecasting or computationally intensive optimization, we propose a lightweight online optimization method to jointly predict a dynamic peak demand target and state-of-charge (SoC) reserve trajectory for peak shaving. This enables real-time adaptive control under evolving demand conditions, resulting in faster and more efficient battery operation. The key contributions of this work are summarized as follows:
\begin{enumerate}
    \item \textbf{Decoupled Control Formulation:} We show that the stacked service control problem can be reformulated as a two-stage optimization by determining the SoC reserve required for peak shaving first, and then the remaining battery capacity is subsequently allocated to other services.
    \item \textbf{Kernel Regression-Based Prediction:} We develop a data-driven kernel regression algorithm that jointly predicts peak-shaving targets and SoC reserve requirements directly from historical demand and battery characteristics.
    \item \textbf{Joint Peak Shaving and Arbitrage:} We integrate the kernel regression predictions with an arbitrage optimization algorithm, enabling decoupled management of peak shaving and energy arbitrage, reducing total electricity costs and mitigating unnecessary battery cycling.
    \item \textbf{Validating Performance through Simulation:} The proposed framework is evaluated using 6 years of electricity demand data from a 35-story commercial building in New York City, demonstrating improved peak shaving and cost savings compared to a state-of-the-art method.
\end{enumerate}

The remainder of the paper is organized as follows. Section~\ref{formulation} introduces the formulation, and Section~\ref{numerical_methods} details the training and the real-time control algorithm. Section~\ref{simulation} shows the simulation and Section~\ref{conclusion} concludes the paper.

\section{Problem Formulation and Decomposition}\label{formulation}

\subsection{Peak Shaving and Arbitrage Formulation}



We consider a commercial building with a co-located, behind-the-meter BESS operating under a time-varying electricity tariff. The objective is to minimize the overall electricity cost by leveraging both energy arbitrage and peak shaving strategies. 
The deterministic version of the joint peak shaving and arbitrage problem is listed as follows:


\begin{subequations}\label{PS_and_Arb_Formulation}
\begin{align}
F_{\text{comb}}(x) &= \min_{p,d,q} \sum\nolimits_{t=1}^T 
\left[ c d_t + \lambda_t (D_t - d_t + q_t) \right] + \kappa p
\label{det_obj} \\
\text{s.t. } \quad
& p \ge D_t - d_t + q_t, \quad \forall t = 1, \dots, T
\label{peak_def} \\
& e_{t+1} = e_{t} - {d_t}/{\eta} + \eta q_t, \quad \forall t = 1, \dots, T
\label{soc_update} \\
& 0 \le d_t, q_t \le \overline{P}, \quad \forall t = 1, \dots, T
\label{power_bounds} \\
& \underline{E} \le e_{t+1} \le \overline{E}, \quad \forall t = 1, \dots, T
\label{energy_bounds}
\end{align}
\end{subequations}

The decision variables include the discharge energy per timestep ($d_t$), charge energy per timestep ($q_t$), and peak demand within the optimization horizon ($p$). The battery’s maximum energy capacity and energy capacity per timestep are denoted by $\overline{E}$ and $\overline{P}$, respectively. $\underline{E}$ denotes the minimum SoC bound and $\eta$ denotes the efficiency for charging and discharging. The objective function in \eqref{det_obj} minimizes total cost by balancing peak demand charges ($\kappa$), real-time energy price ($\lambda$), and battery degradation cost ($c$). The model includes constraints for battery behavior and peak shaving. \eqref{peak_def} defines $p$ as the maximum net demand. SoC dynamics are updated by \eqref{soc_update}, and \eqref{power_bounds} and \eqref{energy_bounds} enforce power and energy limits.

\subsection{Decomposition of Stacked Services}
The deterministic formulation of the problem in \eqref{PS_and_Arb_Formulation} is impractical to solve in practice, as both demand and electricity prices are uncertain, while the peak demand charge is typically settled on a monthly basis. Solving the problem over such a long uncertainty horizon is therefore computationally intractable. To address this challenge, based on the fact that the peak penalty $\kappa$ is much larger than the energy cost coefficients $c$ and $\lambda_t$ in practice, \eqref{PS_and_Arb_Formulation} can be effectively decomposed into two sequential subproblems. This decomposition first prioritizes minimizing peak demand and then performs energy arbitrage over the resulting feasible set.


\textbf{Peak Shaving.} 
In the first stage, we minimize the peak demand and determine the associated trajectory of the SoC reserve necessary for peak shaving. In comparison to the objective in \eqref{PS_and_Arb_Formulation}, we introduce a small tie-breaking term $\delta \sum_{t=1}^T e_{t+1}$ in the objective. This term keeps SoC levels low but does not affect the peak value when $\delta$ is sufficiently small. Solving \eqref{PS_Only_Formulation} thus yields an optimal peak demand $p^*$ and SoC profile $e^*$. 

\begin{equation}\label{PS_Only_Formulation}
\begin{aligned}
& F_{\text{PS}}(x;\delta) = \min_{p,d,q,e} \; p + \delta \sum\nolimits_{t=1}^T e_{t+1} \\
\text{s.t. } \quad
& \text{Constraints \eqref{peak_def}--\eqref{energy_bounds}}
\end{aligned}
\end{equation}


\textbf{Arbitrage.} 
Given the Stage 1 solution ($p^*$, $e^*$), Stage 2 maximizes the revenue from energy arbitrage while ensuring that arbitrage decisions respect both the peak shaving constraint $p \le p^*$ and the SoC trajectory established in Stage 1.
\begin{subequations}\label{Arb_Formulation_Stage2}
\begin{align}
& F_{\text{arb}}(x) = \min_{d,q,e} \sum\nolimits_{t=1}^T \big[ c d_t - \lambda_t (d_t - q_t) \big]
\label{arb_obj_min} \\
\text{s.t. } \quad
& p^* \ge D_t - d_t + q_t, 
\quad \forall t = 1,\dots,T
\label{peak_const_stage2} \\
& e_{t+1} \ge e_{t+1}^*, 
\quad \forall t = 1,\dots,T
\label{PS_SOC_min_stage2} \\
&  \eqref{soc_update}-\eqref{energy_bounds}. \nonumber
\end{align}
\end{subequations}






We now present the following proposition proving that the decomposition is equivalent to the original problem:

\begin{proposition}\label{Equivalence_Prop}
    Given that $\delta$ is sufficiently small and $\kappa \gg \lambda_t, c$, solving the two-stage optimization in \eqref{PS_Only_Formulation} and \eqref{Arb_Formulation_Stage2} is equivalent to solving the combined peak shaving and arbitrage problem in \eqref{PS_and_Arb_Formulation}.
\end{proposition}

\begin{proof}
We first restate the combined problem for convenience:
\begin{equation}
    F_{\text{comb}}(x)=\sum\nolimits_{t=1}^T \lambda_t D_t + F_{\text{arb}}(x) + \kappa F_{\text{PS}}(x). \nonumber
\end{equation}
Let $x=(p,\{d_t,q_t,e_{t+1}\}_{t=1}^T)$ denote the decision variables, and let $S$ be the feasible region defined by \eqref{peak_def}–\eqref{energy_bounds}, where $g(x)\le0$ collects all inequality constraints and $h(x)=0$ the equality constraint.  

The Lagrangian of the combined problem is
\begin{equation}
    \mathcal{L}_{\text{comb}}(x,\mu,\nu)=F_{\text{arb}}(x)+\kappa F_{\text{PS}}(x)+\mu^T g(x)+\nu^T h(x),
\end{equation}
with $\mu\ge0$ and $\nu$ denoting the dual variables. The KKT stationarity condition is
\begin{equation}\label{eq:kkt_combined}
    \nabla_x F_{\text{arb}}(x)+\kappa\nabla_x F_{\text{PS}}(x)+\nabla g(x)^T\mu+\nabla h(x)^T\nu=0,
\end{equation}
accompanied by primal and dual feasibility and complementary slackness.  

To analyze the limit $\kappa \gg \lambda_t, c$, define rescaled dual variables $\tilde{\mu}=\mu/\kappa$ and $\tilde{\nu}=\nu/\kappa$. Substituting into \eqref{eq:kkt_combined} and dividing by $\kappa$ yields
\begin{equation}\label{eq:kkt_rescaled}
   \nabla_x F_{\text{arb}}(x)/\kappa+\nabla_x F_{\text{PS}}(x)+\nabla g(x)^T\tilde{\mu}+\nabla h(x)^T\tilde{\nu}=0.
\end{equation}
Taking the limit as $\kappa\to\infty$ removes the first term, giving
\begin{equation}\label{eq:kkt_ps_stage1}
    \nabla_x F_{\text{PS}}(x)+\nabla g(x)^T\tilde{\mu}+\nabla h(x)^T\tilde{\nu}=0,
\end{equation}
which are precisely the KKT conditions of the peak shaving problem. Thus, for $\kappa\gg\lambda_t, c$ and $\delta\to0^+$, the solution to \eqref{PS_and_Arb_Formulation} coincides with the peak shaving optimum: $p_{\text{comb}}^*=p^*$ and $e_{\text{comb}}^*=e^*$.  

Since the combined solution minimizes $F_{\text{PS}}(x)$ under these conditions, it lies in the optimal peak shaving set $S^*=\{x\in S\mid F_{\text{PS}}(x)=F_{\text{PS}}^*\}$. Restricting the combined problem to $S^*$, we have:
\[
    \min_{x\in S^*}\big[F_{\text{arb}}(x)+\kappa F_{\text{PS}}(x)\big].
\]
As $F_{\text{PS}}(x)=F_{\text{PS}}^*$ for all $x\in S^*$, the constant term $\kappa F_{\text{PS}}^*$ can be dropped, leaving $\min_{x\in S^*}F_{\text{arb}}(x)$, identical to the arbitrage formulation \eqref{Arb_Formulation_Stage2} over the feasible set characterized by $p\le p^*$ and $e_{t+1}\ge e_{t+1}^*$ for all $t$.  

Therefore, when $\kappa$ is sufficiently large, solving the combined problem \eqref{PS_and_Arb_Formulation} is equivalent to the two-stage procedure of first solving the peak shaving problem \eqref{PS_Only_Formulation} to obtain $(p^*,e^*)$, and then optimizing arbitrage over $S^*$.  \end{proof}

In this work, our primary contribution is the development and evaluation of a peak shaving strategy, including its integration with other stacked services. Energy arbitrage is handled using an established non-anticipatory policy from \cite{baker2023transferable}.

\section{Learning and Control Method}\label{numerical_methods}
We now present our main methods for deriving a non-anticipatory algorithm to solve the joint peak shaving and arbitrage problem using the decomposition approach. Our general approach is as follows: 1) we use historical demand data to generate the SoC reserve and daily net demand targets derived from monthly peak shaving results; 2) we use a kernel regression algorithm with a look-back demand window to jointly predict the SoC reserve target and the peak shaving target; 3) we input the predicted SoC reserve target into a real-time peak shaving and arbitrage controller.





\subsection{Kernel Regression for SoC Reserve Prediction}\label{kernel_regression}

The kernel regression model predicts the minimum SoC reserve required for peak shaving at each timestep by comparing the current demand pattern to historically observed demand trajectories. We first use historical demand data and solve \eqref{PS_Only_Formulation} to generate the hindsight-optimal SoC reserve target $e^{\text{hist}}_s$ and the net demand for each timestep. From the latter, we extract the maximum value within each day, which we define as the hindsight-optimal peak shaving target $p^{\text{hist}}_s$. Note that we use $s$ to represent the time index in the training data and $t$ in the real-time control to reduce confusion of the time index. 

\textbf{Training data generation.} Let the full historical demand record be \(D^{\text{hist}} = \{D^{\text{hist}}_1,\dots,D^{\text{hist}}_N\}\), where \(N\) is the total number of historical timesteps. Let $T$ be the length of the sliding demand windows. Demand look-back vectors are constructed as $\mathbf{D}^{\text{hist}}_s = [D^{\text{hist}}_{s-T+1}, \dots, D^{\text{hist}}_s]$, where each window endpoint $s$ is associated with a corresponding next minimum SoC reserve $e^{\text{hist}}_{s+1}$ and peak shaving target $p^{\text{hist}}_{s+1}$ obtained from the prior training results. Hence each training data entry can be represented as
\begin{align}
    \{e^{\text{hist}}_{s+1}\, , \ p^{\text{hist}}_{s+1} \ | \ \mathbf{X}^{\text{hist}}_s = [\, \mathbf{D}^{\text{hist}}_s,\, t_{\sin,s},\, t_{\cos,s} \,] \}
\end{align}
where $t_{\sin,t}$ and $t_{\cos,t}$ are sine and cosine vectors with a daily period to capture the time-of-the-day feature.

\textbf{Prediction Feature.} In prediction, we assemble the input feature similarly to $\mathbf{X}^{\text{hist}}_s$. Let the observed demand sequence up to the current timestep \(t\) be \(D = \{D_1,\dots,D_t\}\), and at each timestep $t \ge T$, the most recent sequence of $T$ demand values is represented as $\mathbf{D}_t = [D_{t-T+1}, \dots, D_t]$. The prediction feature vector is thus assembled as
\begin{equation}
\mathbf{X}_t = [\, \mathbf{D}_t,\, t_{\sin,t},\, t_{\cos,t} \,],
\end{equation}

\textbf{$\alpha$-Confidence Kernel Regression.} For each current feature vector $\mathbf{X}_t$, the set of $K$ most similar historical vectors, denoted $\mathrm{KNN}(t)$, is identified using Euclidean distance.  
The similarity between $\mathbf{X}_t$ and each neighbor $\mathbf{X}^{\text{hist}}_s$ is quantified using a Gaussian kernel, and the resulting weights are normalized across the $K$ nearest neighbors so that they sum to one:

\begin{equation}\label{kernel_weights}
w_{t,s} = 
\frac{\exp \!\left( -\| \mathbf{X}_t - \mathbf{X}^{\text{hist}}_s \|_2^2 / (2T\sigma^2) \right)}
{\sum_{j \in \mathrm{KNN}(t)} \exp \!\left( -\| \mathbf{X}_t - \mathbf{X}^{\text{hist}}_j \|_2^2 / (2T\sigma^2) \right)}.
\end{equation}

\noindent where $s \in \mathrm{KNN}(t)$ and $\sigma$ is the kernel bandwidth parameter that controls how rapidly the influence of historical windows decays with dissimilarity to the current feature profile.

The resulting weights $w_{t,s}$ are used to compute the predicted required SoC reserve at each timestep. We introduce a confidence-level parameter $\alpha \in (0,1)$ to provide a tunable conservativeness ensuring the predicted SoC reserve is greater than $\alpha$ ratio of the training data. The predicted SoC for the next time step $t+1$ is therefore
\begin{subequations}
\begin{align}
    \hat{e}^{\alpha}_{t+1} &= e^{\text{hist}}_{k} + \frac{\alpha-W_{t,k}}{W_{t,k}-W_{t,k-1}}(e^{\text{hist}}_{k}-e^{\text{hist}}_{k-1}) 
\end{align}
where
\begin{align}
    & W_{t,k-1} \leq \alpha \leq W_{t,k} \\
    &  W_{t,k} = \textstyle \sum_{i=1}^k w_{t,k} \\
    & \text{k is the ascending rank of $e^{\text{hist}}_{s+1}$}, s \in \mathrm{KNN}(t)
\end{align}
\end{subequations}

The historical SoC reserve values of the $K$ nearest neighbors are sorted in ascending order, and their corresponding weights are summed. The SoC value at which the cumulative weight first reaches $\alpha$ is selected as the predicted SoC reserve, allowing the estimate to be adjusted to be more conservative or aggressive depending on operational preferences.

Unlike the SoC reserve, the peak shaving target $p_{t+1}^{pred}$ is computed as a simple weighted average:

\begin{equation}
p_{t+1}^{\text{pred}} = \sum_{s \in KNN(t)} w_{t,s}\, p^{\text{hist}}_s, \quad s \in KNN(t).
\end{equation}

No confidence-level adjustment is applied here because the real-time controller described in Section \ref{sec:rtc} dynamically adapts peak-shaving targets based on the observed net demand, ensuring conservativeness is enforced through the adaptive control rather than through the kernel regression itself.




\subsection{Real-Time Control}\label{sec:rtc}
\begin{algorithm}[htbp]\label{alg:rt_control}
\caption{Real-Time Peak Shaving and Stacked Services Control}
\SetAlgoLined
\SetAlgoNoEnd
\KwIn{$D_t$, $e_t$, $\hat{e}^{\,\alpha}_{t+1}$, $p_{t-1}$, $p^{pred}_{t+1}$, $(\overline{P}, \overline{E}, \eta)$, $q_{t,arb}$, $d_{t,arb}$ \cite{baker2023transferable}.}
\KwOut{$q_t$, $d_t$, $e_{t+1}$, $p_{t+1}$, $D^{net}_t$.}

\textbf{Stage 1: Peak Shaving}\\
$\Delta e_t = \hat{e}^{\,\alpha}_{t+1} - e_t$; $\quad p_t = \max\{p_{t-1}, p^{pred}_{t+1}\}$\\
\If{$D_t > p_t$}{
  $d_{t,\text{PS}} = \min(\max(-\Delta e_t \eta, D_t - p_t), \overline{P}, e_t \eta)$; 
  
  $q_{t,\text{PS}} = 0$
}
\ElseIf{$e_t < \hat{e}^{\,\alpha}_{t+1}$ \text{ and } $D_t < p_t$}{
  $q_{t,\text{PS}} = \min(\Delta e_t / \eta, p_t - D_t, \overline{P}, (\overline{E}-e_t)/\eta)$; 
  
  $d_{t,\text{PS}} = 0$
}
\Else{$q_{t,\text{PS}} = d_{t,\text{PS}} = 0$}
$e_{t+1,\text{PS}} = e_t - d_{t,\text{PS}}/\eta + q_{t,\text{PS}}\eta$; 
$p_{t+1} = \max\{p_t, D^{net}_t\}$\\

\textbf{Stage 2: Arbitrage}\\
$q_{t,\max} = \min(p_{t+1}-D_t, \overline{P}, (\overline{E}-e_t)/\eta)$; 

\noindent $d_{t,\max} = \min((e_t - e_{t+1,\text{PS}})\eta, \overline{P})$\\
\If{$q_{t,\text{PS}} > 0$}{
  $q_t = \min(q_{t,\text{arb}} + q_{t,\text{PS}}, q_{t,\max})$; $d_t = 0$
}
\ElseIf{$d_{t,\text{PS}} > 0$}{
  $d_t = \min(d_{t,\text{arb}} + d_{t,\text{PS}}, d_{t,\max})$; $q_t = 0$
}
\Else{
  $q_t = \min(q_{t,\text{arb}}, q_{t,\max})$; $d_t = \min(d_{t,\text{arb}}, d_{t,\max})$
}
$e_{t+1} = e_t - d_t/\eta + q_t \eta$; $\quad D^{net}_t = D_t - d_t + q_t$
\end{algorithm}

Based on the predicted SoC reserve estimates and daily net demand targets, a real-time controller determines the battery charge schedule, discharge schedule, and SoC trajectory to achieve peak shaving while ensuring feasible operation. This SoC trajectory establishes the energy that must be reserved for peak shaving at each time step. The remaining battery capacity can then be allocated to additional services, such as arbitrage. The control logic operates as described by Algorithm \ref{alg:rt_control}.

\subsection{Parameter Search}\label{parameter_search}
The kernel regression model relies on three key hyperparameters—the kernel bandwidth $\sigma$, the look-back window $T$, and the number of nearest neighbors $K$, which directly influence the performance of the combined peak shaving and arbitrage strategy. Performance is evaluated using two metrics: (1) \textit{total cost savings}, defined as the reduction in electricity cost relative to a no-storage baseline, and (2) \textit{annual battery cycles}, calculated as the total discharge energy divided by the nominal battery capacity $\overline{E}$, indicating utilization and wear. 
To systematically identify effective hyperparameters, we implement an automated, tiered search that adaptively tunes each in sequence. The look-back window $T$ is optimized first through a coarse-to-fine search to find the most informative time horizon. Next, the kernel bandwidth $\sigma$ is tuned using a logarithmic-scale exploration followed by linear refinement to strike a balance between accuracy and smoothness. Finally, the number of nearest neighbors $K$ is adjusted through broad and refined searches to maximize performance.

\section{Simulation and Results}\label{simulation}

We conduct a real-world case study of a 35-story, $800,000~\text{ft}^2$, Class A commercial building in New York City. The demand data for this building, provided by Nantum AI\footnote{ \url{https://github.com/emlog9/peak_shaving_kernel_regression}}, is shown in Fig. \ref{demand_data}. Peak demand charges follow Con Edison’s standard rate, which applies a charge of \$42.80/kW to the monthly peak demand from June~1 to September~30 and \$33.50/kW during other months, plus a flat monthly customer charge of \$71. The monthly peak demand is defined as the average of the two highest consecutive 15-minute intervals of demand\footnote{\url{https://www.coned.com/en/accounts-billing/your-bill/time-of-use}}. Arbitrage is performed using real-time electricity prices from the New York Independent System Operator\footnote{\url{https://www.nyiso.com/energy-market-operational-data}}. To account for seasonal patterns in demand and pricing, the kernel regression model considers only historical data from the same season (either June–September or the other months) when predicting SoC reserve and net demand targets. 

\begin{figure}[ht]
    \centering
\includegraphics[width=0.95\columnwidth]{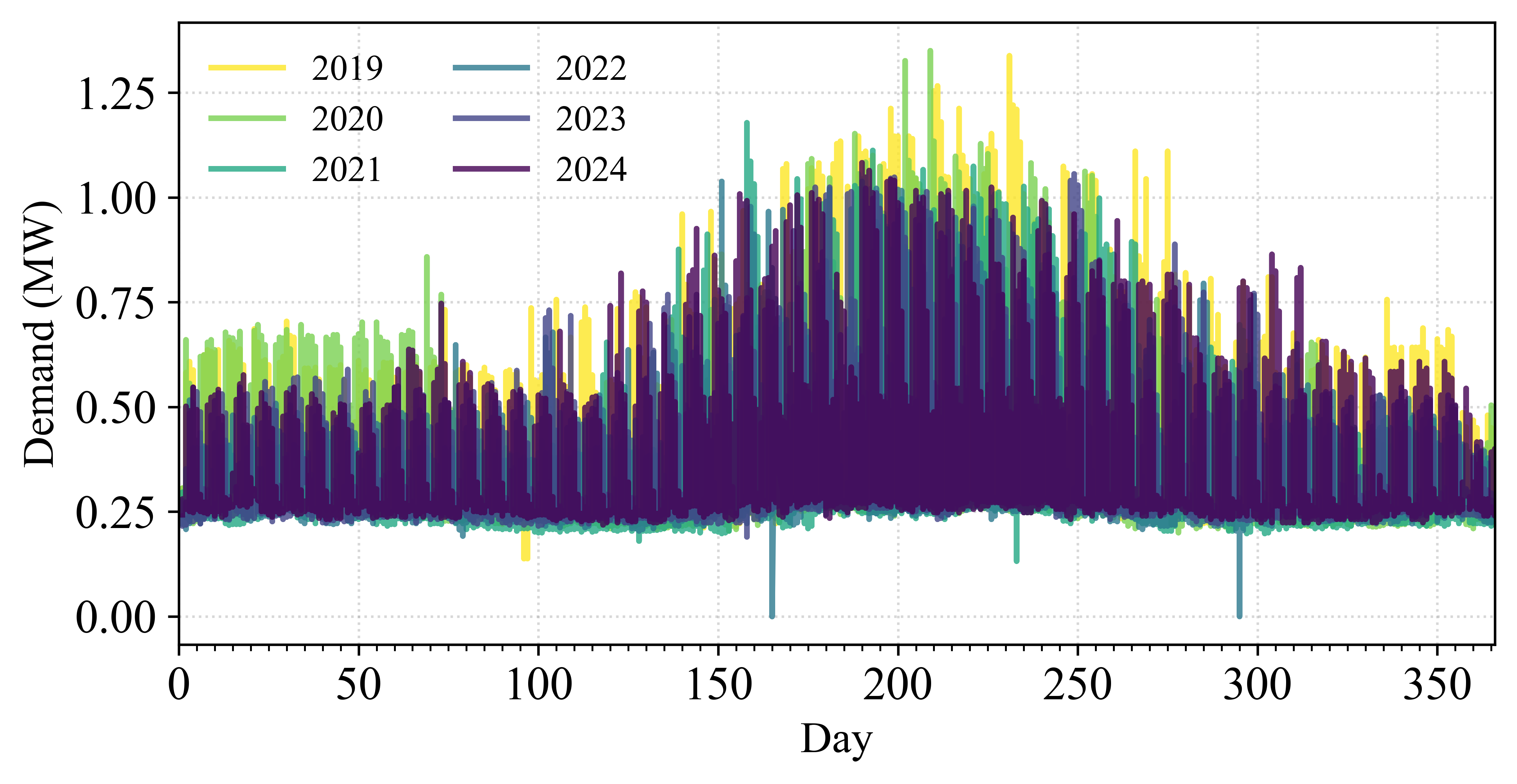}
    \caption{Building electricity demand from 2019–2024. Line colors indicate years. Yearly standard deviations are: 2019: 0.21 MW, 2020: 0.16 MW, 2021: 0.16 MW, 2022: 0.16 MW, 2023: 0.18 MW, and 2024: 0.19 MW.}
    \label{demand_data}\vspace{-0.5cm}
\end{figure}


Historical building data from 2019–2023 is used for model training, and the data from 2024 is reserved for testing. The battery duration ($\overline{E}/\overline{P}$) is 2.2 hours, with an initial state of charge (SoC) of 50\% ($e_0$), a minimum SoC ($\underline{E}$) of 20\% of $\overline{E}$, and an efficiency ($\eta$) of 90\%. Parameters are selected by the automatic algorithm described in Section~\ref{parameter_search}. All simulations are conducted at a 5-minute resolution on a 10-core Apple M4 processor, and all optimization problems are solved using Gurobi~12.0.2. Model performance is evaluated based on total cost savings compared to the no-storage scenario and battery cycling. The results for peak shaving and arbitrage are compared against two benchmarks: (i) the deterministic case outlined in \eqref{PS_and_Arb_Formulation}, limited to an average of one cycle per day, and (ii) a demand-forecasting controller based on a recursive XGBoost model with a \textbf{mean absolute error (MAE) of 0.03 MW}. This model utilizes calendar features, temperature, and lagged demand to generate day-ahead demand forecasts, which are then input into the peak shaving optimization from \eqref{PS_Only_Formulation}. From there, the real-time controller determines the charge and discharge schedule. As with the kernel regression framework, these results are subsequently combined with an arbitrage schedule from \cite{baker2023transferable} in a two-stage approach. Simulations are performed for battery power ratings ranging from 0.1~MW to 1.0~MW.


\begin{figure}[ht]
    \centering
\includegraphics[width=0.95\columnwidth]{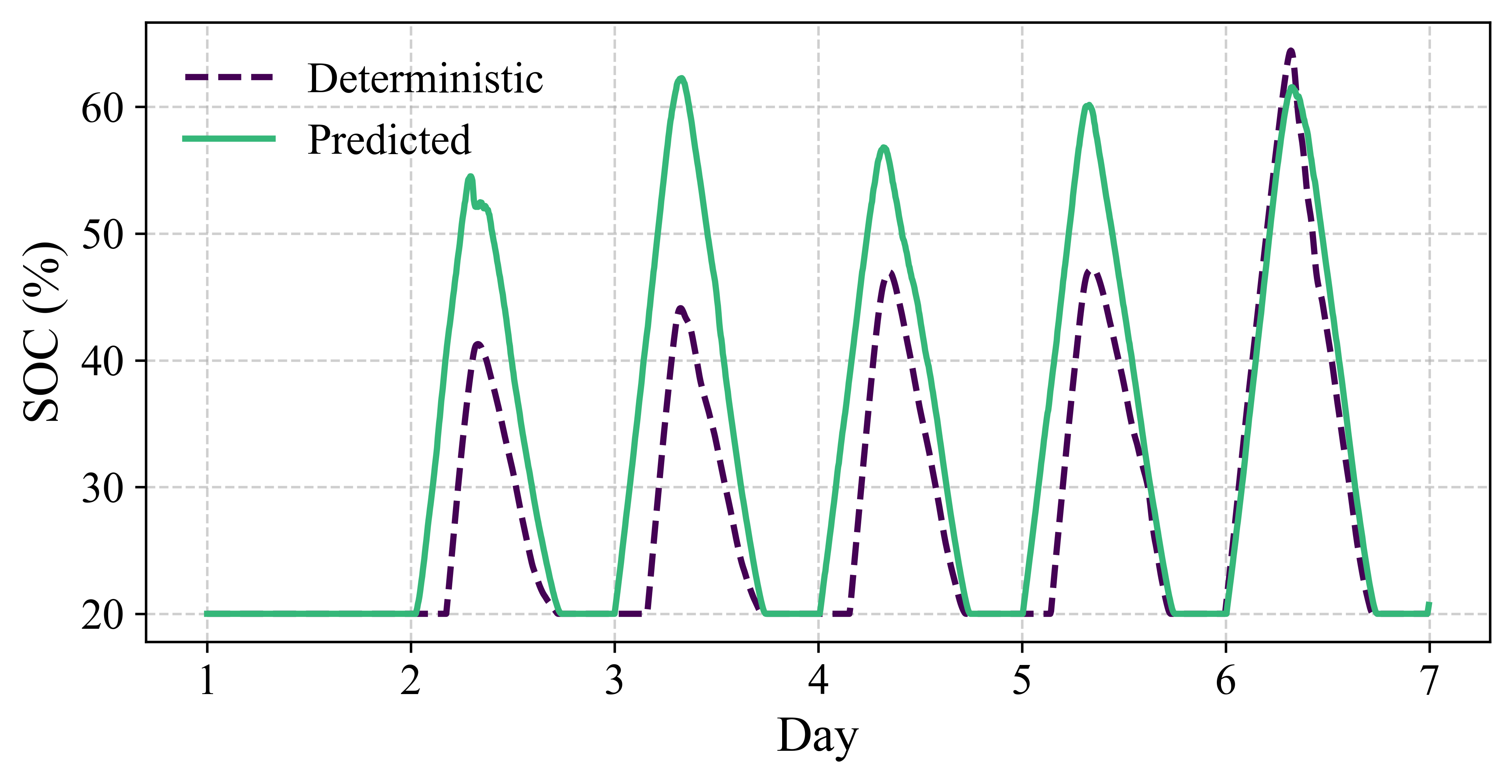}
    \caption{Comparison of the kernel-regression predicted and deterministic minimum SoC reserve trajectories for a 1 MW battery.} 
    \label{fig:SOCs}
\end{figure}

Fig.~\ref{fig:SOCs} compares the predicted SoC reserve trajectory to the deterministic SoC trajectory for a representative week. The predicted SoC generally tracks the timing and magnitude of the actual SoC peaks, though it tends to slightly overestimate the reserve. This behavior aligns with the design of the confidence level parameter ($\alpha$), which accounts for uncertainty by allowing for more conservative SoC predictions.


Fig. \ref{fig:peak_demand} shows the peak demand across the three controller types in comparison to the peak demand with no battery storage. Overall, the kernel regression controller effectively shaves the peak demand, outperforming the demand forecasting model in most months.

\begin{figure}[ht]
    \centering
\includegraphics[width=0.95\columnwidth]{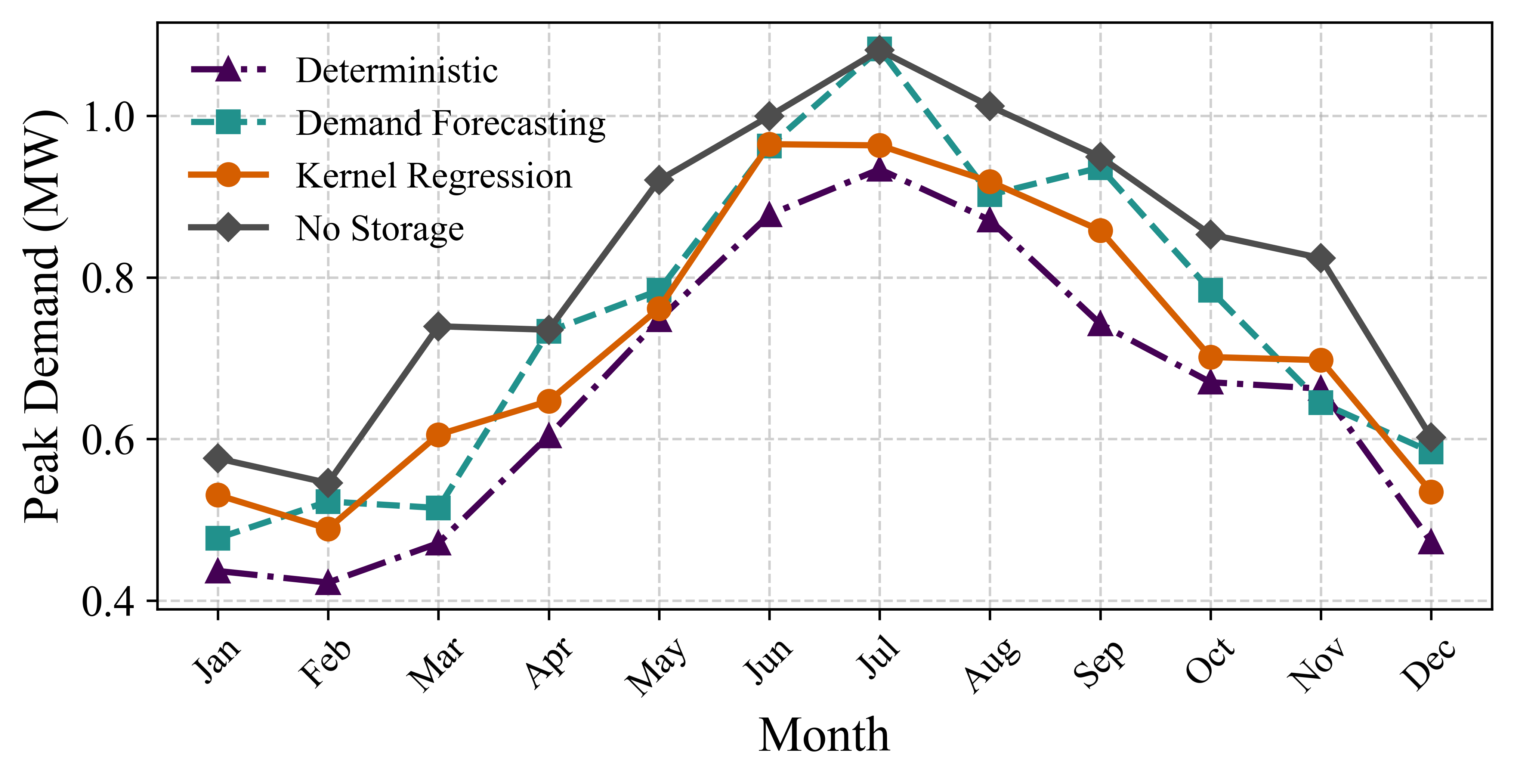}
    \caption{Peak demand (MW) by month for perfect foresight, kernel regression-based, and demand forecasting control algorithms for a 0.6 MW battery.}
    \label{fig:peak_demand}
\end{figure}


To evaluate the economic impact of the proposed control strategy, we analyze annual cost savings relative to a scenario with no storage, where the total electricity cost is \$512.5k, as a function of battery power capacity (MW). This metric captures the economic benefit of combined peak shaving and energy arbitrage given real-time electricity prices and demand patterns. Fig.~\ref{fig:savings} compares the results for the kernel regression–based controller with the deterministic and demand-forecasting baseline cases.

\begin{figure}[ht]
    \centering
\includegraphics[width=0.95\columnwidth]{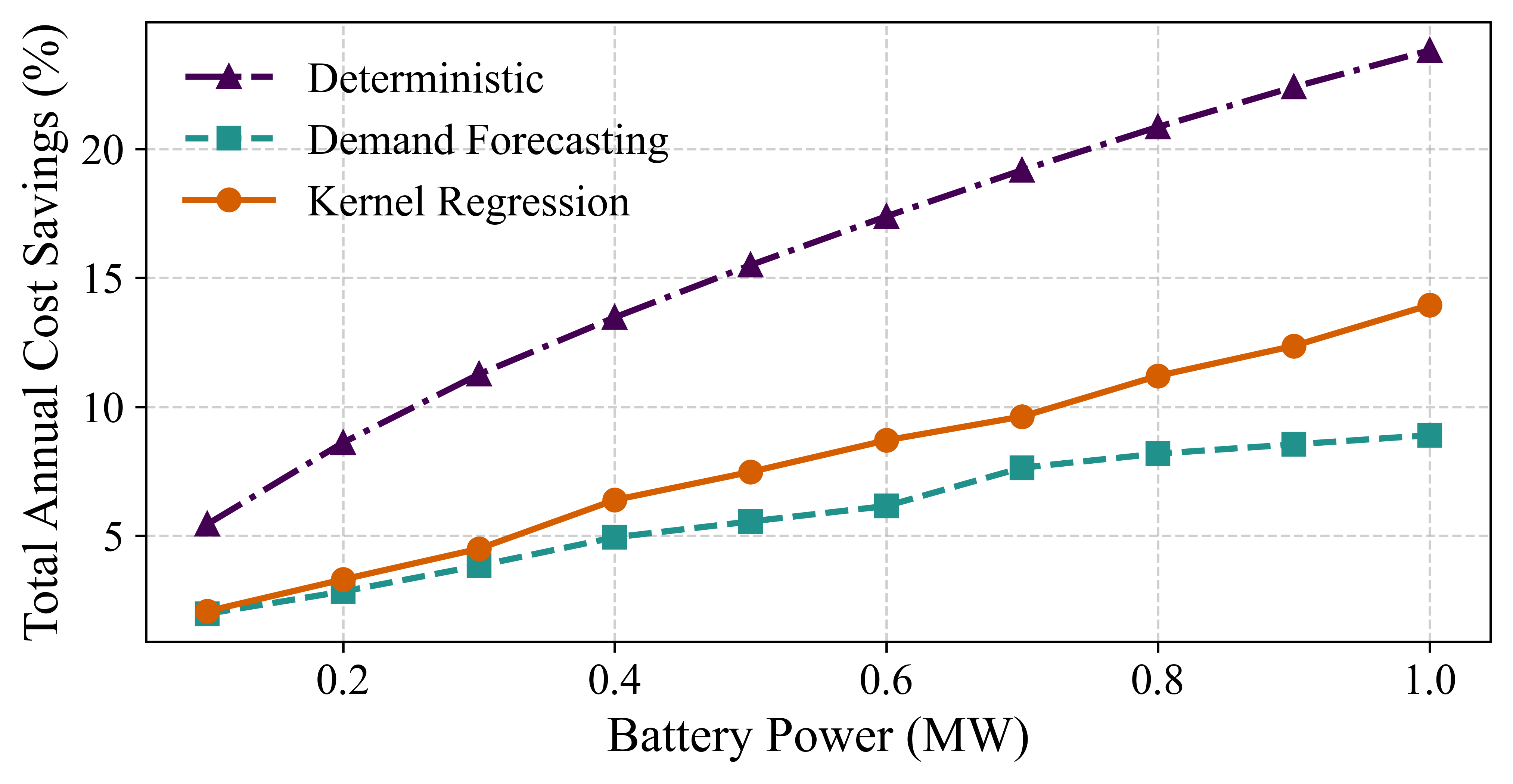}
    \caption{Total annual cost savings (\%) vs. battery size for perfect foresight, kernel regression-based, and demand forecasting control algorithms.}
    \label{fig:savings}
\end{figure}


Larger batteries achieve higher savings, although the increasing trend saturates. Across the battery sizes evaluated in this study, the kernel regression model captures between 38.1\% and 58.5\% of the deterministic cost savings, outperforming the demand forecasting controller for each battery size, which achieves only 33.0\% to 39.8\% of the deterministic cost savings. It is worth noting that the deterministic perfect foresight assumes one month of perfect demand and price forecast (based on a monthly utility bill settlement). Consequently, the performance of the perfect forecast is significantly better. 




Battery cycling is an important metric for understanding battery utilization and wear. 
The deterministic controller operates near its cycle limit of one cycle per day on average, ranging from 346 to 366 cycles annually across the range of battery power capacities (0.1 MW to 1 MW). This behavior indicates that with perfect foresight, maximizing cost savings and fully exploiting peak shaving and arbitrage opportunities requires aggressive daily cycling.
In contrast, the kernel regression controller exhibits more variable cycling, ranging from approximately 76 to 126 cycles per year across battery sizes. This variation arises because the controller relies on predicted SoC trajectories rather than perfect knowledge of future demand and prices. Underestimation or overestimation of the SoC can cause missed opportunities to charge or discharge at optimal times. The demand forecasting model cycles even less, ranging from 57 to 95 annual cycles.


\section{Conclusion}\label{conclusion}

This work presents an end-to-end framework for behind-the-meter peak shaving and stacked energy services in commercial buildings. The method employs a two-stage decomposition, where a nonparametric kernel regression model generates a state-of-charge reserve trajectory for real-time peak shaving control without requiring explicit demand forecasts. The remaining battery capacity is then allocated to energy arbitrage via a non-anticipatory control strategy. A real-world case study demonstrates that the proposed approach outperforms a standard demand forecasting baseline, capturing 38.1\%–58.5\% of the cost savings achieved by a perfect-foresight benchmark. The results highlight the method’s practicality and economic effectiveness for real-time battery management. For future work, we aim to enhance the kernel regression model and integrate additional services such as frequency regulation.

\bibliographystyle{IEEEtran}
\bibliography{IEEEabrv,PS}

\end{document}